\documentclass{amsart}
\usepackage[english]{babel}
\usepackage[latin1]{inputenc}
 \usepackage[all]{xy}
 \usepackage[pagebackref]{hyperref}

\usepackage{amsmath,amsfonts,amssymb,amsthm,amscd,array,stmaryrd,mathrsfs, mathdots}
\usepackage[makeroom]{cancel}
\PassOptionsToPackage{option}{xcolor}
\usepackage{pstricks}
\usepackage{tikz}
\usepackage{graphicx}

\theoremstyle{plain}

\newtheorem{thm}{Theorem}

\newtheorem{lem}{Lemma}

\newtheorem{prop}[lem]{Proposition}
\newtheorem{conj}{Conjecture}

\theoremstyle{definition}
\newtheorem{defn}[lem]{Definition}
\newtheorem{rem}[lem]{Remark}
\newtheorem{ex}[lem]{Example}

\let\ssection=\section
\renewcommand{\section}{\setcounter{equation}{0}\ssection}

\newcommand{\R}{\mathbb{R}}
\newcommand{\Z}{\mathbb{Z}}
\newcommand{\C}{\mathbb{C}}

\newcommand{\Q}{\mathbb{Q}}

\newcommand{\B}{\mathcal{B}}

\newcommand{\Rc}{\mathcal{R}}
\newcommand{\Sc}{\mathcal{S}}
\newcommand{\Vc}{\mathcal{V}}
\newcommand{\Uc}{\mathcal{U}}

\newcommand{\Id}{\mathrm{Id}}
\newcommand{\GL}{\mathrm{GL}}

\newcommand{\PGL}{\mathrm{PGL}}
\newcommand{\SL}{\mathrm{SL}}
\newcommand{\PSL}{\mathrm{PSL}}

\def\b{\beta}

\def\s{\sigma}

\begin{document}

\title[Burau representation of $\B_3$]
{Burau representation of braid groups
and $q$-rationals}

\author[S. Morier-Genoud]{Sophie Morier-Genoud}
\author[V. Ovsienko]{Valentin Ovsienko}
\author[A. Veselov]{Alexander Veselov}

\address{Sophie Morier-Genoud,
Laboratoire de Math\'ematiques,
Universit\'e de Reims,
U.F.R. Sciences Exactes et Naturelles,
Moulin de la Housse - BP 1039,
51687 Reims cedex 2,
France
}

\address{
Valentin Ovsienko,
Centre National de la Recherche Scientifique,
Laboratoire de Math\'ematiques,
Universit\'e de Reims,
U.F.R. Sciences Exactes et Naturelles,
Moulin de la Housse - BP 1039,
51687 Reims cedex 2,
France}

\address{
Alexander P. Veselov,
Department of Mathematical Sciences, 
Loughborough University, 
Loughborough, LE11 3TU, UK}

\email{sophie.morier-genoud@univ-reims.fr,
valentin.ovsienko@univ-reims.fr,
A.P.Veselov@lboro.ac.uk}

\maketitle

\begin{abstract}

We establish a link between the new theory of $q$-deformed rational numbers 
and the classical Burau representation of the braid group $\B_3$.
We apply this link to the open problem of classification of faithful complex specializations of this representation.
As a result we provide an answer to this problem in terms of the singular set of the $q$-rationals and prove the faithfulness of the Burau representation specialized at complex $t\in \C^*$ outside the annulus $3-2\sqrt2 \leq |t| \leq 3+2\sqrt2.$

\end{abstract}

\thispagestyle{empty}

\section{Introduction}

The braid groups are the most remarkable groups from topological point of view naturally appearing in the theory of knots, mapping class groups and configuration spaces, see \cite{Bir}.
They were explicitly introduced by Emil Artin in 1925 \cite{Artin}, who denoted their $n$-strand versions by $\B_n$. 
 The (Artin) braid group $\B_n$ is generated by $n-1$ elements $\sigma_1, \dots, \sigma_{n-1}$ with braid relations
$$
\sigma_i\sigma_{i+1}\sigma_i=\sigma_{i+1}\sigma_{i}\sigma_{i+1}, \quad i=1,\dots, n-1,
$$
and $\sigma_i \sigma_j=\sigma_j \sigma_i$ when $|i-j|>1.$
Since then the braid groups have been extensively studied both by topologists and algebraists. 

One of the first important results in this direction was found by Werner Burau in 1936 \cite{Burau}, who introduced what is now known as (reduced) {\it Burau representation}
$\rho_n: \B_n \to \GL(n-1, \Z[t,t^{-1}]).$ In the simplest case $n=3$
the Burau representation $\rho_3: \B_3 \to \GL(2, \Z[t,t^{-1}])$ is defined by
\begin{equation}
\label{rho3}
\begin{array}{rclrcl}
\rho_3\; : \quad \sigma_1 &\mapsto& 
\begin{pmatrix}
-t&1\\[4pt]
0&1
\end{pmatrix}, & \sigma_2 &\mapsto& \begin{pmatrix}
1&0\\[4pt]
t&-t
\end{pmatrix},
\end{array}
\end{equation}
where $t$ is a formal parameter.
One can check that the braid relation $\sigma_1\sigma_2\sigma_1=\sigma_2\sigma_1\sigma_2$ is satisfied in this case.
Burau used this representation to introduce a knot invariant, which turned out to be essentially the famous 
Alexander polynomial~\cite{Alexander}. 

It is known after~\cite{Arnold,MaPe} that $\rho_3$ is {\it faithful}, meaning that its kernel is trivial.
Three different proofs of this fact can be found in~\cite{BB}. Note that for $n\geq 5$
the Burau representation is known to be non-faithful \cite{LP, Bigelow} and that for $n=4$ the question is still open.

In this paper we address another open problem (which we call {\it Burau specialization problem}), 
which was flagged in a recent paper by Bharathram and Birman (see~\cite[Open Problem 1 in Section 7]{BB}):


{\it ``At which complex specializations of $t$ is the Burau representation $\rho_3$ faithful?"}

Here a {\it specialization} of the  Burau representation $\rho_3$ is the representation
\begin{equation*}
\label{rho3t}
\rho^t_3: \B_3 \to \GL(2, \C)
\end{equation*}
which is defined by~\eqref{rho3},
but $t$ is now a non-zero complex number: $t\in\C^*=\C\setminus\{0\}$.

In the real case with $t\in\R$ some interesting results in this direction were found by Scherich in~\cite{Sch}, who used the hyperbolic geometry to prove, in particular, that $\rho^t_3$ is faithful when $t<0,\, t\not=-1$,
and outside the interval $\frac{3-\sqrt5}{2}\leq t\leq\frac{3+\sqrt5}{2}$.
Note that when $t=-1$ the matrices (\ref{rho3}) specialize to 
$$
R:=\begin{pmatrix}
1&1\\[4pt]
0&1
\end{pmatrix},  
\qquad
L^{-1}:=\begin{pmatrix}
1&0\\[4pt]
-1&1
\end{pmatrix},
$$
generating the whole group $\SL(2, \Z).$ Recall also that the quotient of the braid group~$\B_3$ 
by its centre $Z$ (generated by $z:=(\sigma_1\sigma_2)^3$) is the classical modular group 
$\PSL(2, \Z)=\SL(2, \Z)/\pm \Id,$ so that the specialization of the Burau representation~$\rho_3$ 
at $t=-1$ is not faithful, having $Z$ in its kernel.

The aim of our work is to establish the link of the Burau specialization problem 
with the theory of $q$-deformed rational numbers (or $q$-rationals, for short), which was
recently developed in~\cite{MGOfmsigma}. 
The $q$-deformation of a positive rational $\frac{r}{s}$ has the form
$$
\left[\frac{r}{s}\right]_q=\frac{\Rc(q)}{\Sc(q)},
$$
where $\Rc(q), \, \Sc(q)$ are certain Laurent polynomials in $q$ with positive integer coefficients (see the details in the next section). The zeros of these polynomials have been studied in \cite{LMGOV}.

Define the {\it singular set of $q$-rationals} $\Sigma\subset \C^*$ as the union of complex poles of all $q$-rationals and consider the 
{\it extended singular set} $\Sigma_*:=\Sigma \cup \{1\}$. 
Both sets $\Sigma$ and $\Sigma_*$ consist of certain algebraic integers and are invariant under the involution $q\to q^{-1}.$

The following theorem gives an answer to the Burau specialization problem in terms of $q$-rationals. 

\begin{thm}
\label{ThmMain1}
The Burau representation $\rho_3$ specialized at $t_0\in  \C^{*}$ is faithful if and only if $-t_0 \notin \Sigma_*.$
\end{thm}


A similar claim in terms of zeros of the Moody polynomials \cite{Moody-1991} can be found in \cite{BB},
but our statement allows us to use the theory of $q$-rationals for which information about the set~$\Sigma$
is available. 
In particular, from the results of \cite{LMGOV} we deduce our second main result.

\begin{thm}
\label{ThmMain2}
The specialized Burau representation $\rho^{t_0}_3$ is faithful for all
$t_0\in  \C^{*}$ outside the annulus  $3-2\sqrt2 \leq |t_0| \leq 3+2\sqrt2$.
\end{thm}



We believe that we can reduce the annulus in Theorem 2 as follows.

\begin{conj}
\label{CoNash}
The specialized Burau representation $\rho^{t_0}_3$ is faithful for all
$t_0\in  \C^{*}$ outside the annulus  $\frac{3-\sqrt5}{2}\leq |t| \leq \frac{3+\sqrt5}{2}$.
\end{conj}

Modulo Theorem 1 this conjecture  is a weaker version of the conjecture from~\cite{LMGOV}, which claims, in particular, that $\Sigma_*$ lies within the annulus $\frac{3-\sqrt5}{2}\leq |q| \leq \frac{3+\sqrt5}{2}.$

Note that the bounds $\frac{3-\sqrt5}{2}$ and $\frac{3+\sqrt5}{2}$ were also found in~\cite{Sch} in the real case,
which confirms this conjecture for $t\in\R^*$.

Our approach is based on the following important observation. Consider the natural projective version of the Burau representation
\begin{equation}
\label{prho3}
\hat\rho_3: \B_3 \to \PGL(2, \Z[t,t^{-1}]):=\GL(2, \Z[t,t^{-1}])/\{\pm t^n \Id, \, n \in \Z\}.
\end{equation}
Note that the centre $Z \subset \B_3$ is in the kernel of $\hat\rho_3$ since 
$\rho_3((\s_1\s_2)^3)=t^3 \Id.$ This means that image $\hat\rho_3(\B_3)$ is isomorphic to the modular group $\PSL(2, \Z).$

Our key observation is that this image coincides with the $q$-deformation $\PSL(2, \Z)_q$ of the modular group $\PSL(2, \Z)$ introduced in~\cite{MGOfmsigma} (see also  \cite{LMGadv}) if we identify $t$ with $-q$. Indeed, $\PSL(2, \Z)_q$ is generated by
$$
\begin{pmatrix}
q&1\\[4pt]
0&1
\end{pmatrix},
\qquad
\begin{pmatrix}
1&0\\[4pt]
1&q^{-1}
\end{pmatrix},
$$
where $q$ is a formal parameter,
which are projectively equivalent to $\rho_3(\s_1)$ and $\rho_3(\s_2)^{-1}$ respectively when $q=-t$.

We should note that the theory of $q$-rationals was initially motivated by connections to cluster algebras 
and rapidly led to further developments in various directions:
combinatorics of posets, 
knot invariants, 
Markov numbers and Diophantine analysis, 
enumerative geometry,  triangulated categories and homological algebra, 
quantum calculus (see more on this in the review~\cite{Oamr}).
It is interesting that although the knot theory was already discussed in this connection,
as far as we know, no relation with the Burau representation $\rho_3$ was pointed out so far.



\section{$q$-rationals and $q$-deformed modular group}\label{RecolSec}

In this section we give the precise and adapted for computations definition of the $q$-rationals,
and recall their main properties, which are necessary for the proof of the main result.
{ Among several equivalent definitions of $q$-rationals we will use the one based} on $\PSL(2,\Z)$-invariance
(or on $\B_3$-invariance).

\subsection{The main definition}

Consider the set $\Q\cup\{\infty\}$ rational numbers, extended by one additional element~$\infty$ 
which will always be represented by the quotient $\frac{1}{0}$.
The group $\SL(2,\Z)$ of matrices with integer coefficients
$$
M=\begin{pmatrix}r&v\\s&u\end{pmatrix},
\qquad 
r,v,s,u\in\Z,
\quad
ru-vs=1,
$$
acts on $\Q\cup\{\infty\}$ by linear-fractional transformations: 
$$
M\cdot x=\frac{rx+v}{sx+u}.
$$
This action is homogeneous and effective for the modular group $\PSL(2,\Z)$.

Following~\cite{MGOfmsigma}, consider the following matrices of $\GL(2, \Z[q,q^{-1}])$:
\begin{equation}\label{RqLq}
R_{q}:=
\begin{pmatrix}
q&1\\[4pt]
0&1
\end{pmatrix},
\qquad
L_{q}:=
\begin{pmatrix}
1&0\\[4pt]
1&q^{-1}
\end{pmatrix}.
\end{equation}
They generate a subgroup $\PSL(2, \Z)_q$ in the group 
$$
\PGL(2, \Z[q,q^{-1}])=\GL(2, \Z[q,q^{-1}])/\{\pm q^n \Id\}.
$$
which is isomorphic to $\PSL(2,\Z)$ and thus can be consi\-dered as {\it $q$-deformed modular group} \cite{MGOfmsigma, LMGadv}.
Evaluating at $q=1$ one recovers $R=R_{q}{|_{q=1}}$ and $L=L_{q}{|_{q=1}}$ the standard generators of the modular group $\SL(2, \Z)$.
This allows one to consider $q$-analogs of matrices $M_q$ for every $M\in\SL(2, \Z)$. 

We have a natural action of $\PSL(2, \Z)_q$ on the space~$\Z(q)$ of rational functions in~$q$ with integer coefficients, which can be considered as $q$-deformed version of the standard linear-fractional action of the modular group.
More precisely, for every $M\in\SL(2,\Z)$, written as a monomial in $R$ and $L$, let $M_q$ be a matrix in $\GL(2, \Z[q,q^{-1}])$
given by the same monomial in $R_q$ and $L_q$.
If~$f(q)\in\Z(q)$, then the linear-fractional action
\begin{equation}
\label{ActPSL}
f(q)\mapsto M_q\cdot{}f(q)
\end{equation}
is a well-defined action of  $\PSL(2,\Z)$.
In other words, $M_q$ does depend on the choice of the monomial in $R$ and $L$ presenting~$M$, but 
the result of~\eqref{ActPSL} does not.

\begin{defn}
\label{MainDef}
The set of $q$-rationals is the set of rational functions in the orbit of { any point from the set $\{0, \, 1, \, \infty\}$
(all of them remain independent of~$q$ after deformation)
for the linear-fractional action~\eqref{ActPSL}.}
\end{defn}

\begin{ex}
For instance,
$$
L_q\cdot\infty=R_q\cdot 0=1, 
\quad R_qR_q\cdot 0= 1+q, 
\quad L_qR_q\cdot 0=\frac{q}{1+q},  
\quad L_qR_qR_q\cdot 0=\frac{q+q^2}{1+q+q^2}, 
\quad \ldots
$$
are considered as the $q$-analogs of respectively $1, 2, \frac12, \frac23, \ldots$
\end{ex}

\begin{rem}
The notion of $q$-rationals extends that of $q$-integers:
\begin{eqnarray*}
[n]_q&:=&1+q+q^2+\ldots +q^{n-1}\\
{[-n]_q}&:=&-q^{-1}-q^{-2}\ldots -q^{-n},
\end{eqnarray*}
that goes back to the works of Euler and Gauss.
These $q$-integers and the corresponding $q$-binomial coefficients are essential in quantum algebra and mathematical physics.
They are also the key ingredients for the theory of $q$-analogs in combinatorics.
 Most classical sequences of integers have interesting $q$-analogs often arising as generating functions.
\end{rem}

\subsection{Matrices of continued fractions}

We use the following notation for the entries and the decomposition in terms of generators of a matrix $M$ in $\SL(2,\Z)$
\begin{equation}\label{rstu}
\begin{array}{lclclcl}
M&=&
\begin{pmatrix}
r&v\\s&u\end{pmatrix}&=&R^{a_1}L^{a_2}\cdots R^{a_{2m-1}}L^{a_{2m}}&=:&M^+(a_1,  \ldots, a_{2m}), \end{array}
\end{equation}
 and the following for the corresponding $q$-deformed matrices in $\GL(2, \Z[q,q^{-1}])$
\begin{equation}\label{rstuq}
\begin{array}{lclclcl}
M_q&=&
\begin{pmatrix}\Rc&\Vc\\\Sc&\Uc\end{pmatrix}&=&R_q^{a_1}L_q^{a_2}\cdots R_q^{a_{2m-1}}L_q^{a_{2m}}&=:&M_q^+(a_1, \ldots, a_{2m}),\end{array}
\end{equation}
where $a_i$ and $r,s,t,u$ are integers, and $\Rc, \Sc, \Vc, \Uc $ are Laurent polynomials in $q$.
This notation is very convenient for explicit computations, 
taking into account that the coefficients $a_i$ are those of the continued fraction
expansion $\frac{r}{s}=[a_1,  \ldots, a_{2m}]$.

The following definition is equivalent to Definition~\ref{MainDef}.

\begin{defn} \cite{MGOfmsigma} 
If $\frac{r}{s}$ is a rational appearing in the first column of a matrix $M$ we define its $q$-analog as the rational function given by
\begin{equation}\label{RatRS}
\left[\frac{r}{s}\right]_q:=\frac{\Rc(q)}{\Sc(q)},
\end{equation}
where $\Rc(q)$ and $\Sc(q)$ are the entries in the first column of the corresponding matrix $M_q$.
\end{defn}

\begin{ex}\label{ex}
From
$$
M^+(1,1)=\begin{pmatrix}
2&1\\[2pt]
1&1
\end{pmatrix}, \quad
M_q^+(1,1)=\begin{pmatrix}
1+q&q^{-1}\\[2pt]
1&q^{-1}
\end{pmatrix}, 
$$
one gets $[2]_q=1+q$, which is the standard quantum integer.

From
$$
M^+(0,2)=\begin{pmatrix}
1&0\\[2pt]
2&1
\end{pmatrix}, \quad
M_q^+(0,2)=\begin{pmatrix}
1&0\\[2pt]
1+q^{-1}&q^{-2}
\end{pmatrix}, 
$$
one gets $\left[\frac12\right]_q=\frac{q}{1+q}$.

From
$$
M^+(0,1,1,1)=\begin{pmatrix}
2&1\\[2pt]
3&2
\end{pmatrix}, \quad
M_q^+(0,1,1,1)=\begin{pmatrix}
1+q&q^{-1}\\[2pt]
q+1+q^{-1}&q^{-1}+q^{-2}
\end{pmatrix}, 
$$
one gets $\left[\frac23\right]_q=\frac{q+q^2}{1+q+q^2}$.

\end{ex}

\subsection{Some properties of the polynomials $\Rc$ and $\Sc$}

 Let as before $\left[\frac{r}{s}\right]_q=\frac{\Rc(q)}{\Sc(q)}$.
We collect some of the known properties of the polynomials $\Rc(q)$ and $\Sc(q)$ that will be useful for the proof.
We were not able to find these properties in the literature about the Burau representation.

Suppose first that $\frac{r}{s}\geq1$.

\begin{prop}[Positivity {\cite[Prop. 1.3]{MGOfmsigma}}]\label{pos}
The polynomials  $\Rc$ and $\Sc$ have positive integer coefficients.
\end{prop}

Note that this positivity statement can be strengthened, cf. Theorem~2 of~\cite{MGOfmsigma}.

\begin{prop}[Reflection and mirror {\cite[Prop 2.8]{LMGadv}}]\label{RefPro}
One has
$$
\left[\frac{s}{r}\right]_q=\frac{\Sc(q^{-1})}{\Rc(q^{-1})} 
\qquad \text{and } \qquad\left[-\frac{r}{s}\right]_q=-\frac{\Rc(q^{-1})}{q\Sc(q^{-1})}.$$
\end{prop}

\begin{prop}[Roots of polynomials {\cite[Section 5.1]{LMGOV}}]\label{root}
The polynomial $\Rc$ and $\Sc$ 
have no roots inside the punctured  disc with radius $3-2\sqrt2$.
\end{prop}

Proposition~\ref{root} can be extended for an arbitrary rational $\frac{r}{s}<1$ using $\PSL(2,\Z)$-action.
Proposition~\ref{RefPro} implies that the polynomials $\Rc$ and $\Sc$ have no roots for 
$|q|\geq\frac{1}{3-2\sqrt2}=3+2\sqrt2$.
Therefore, one gets the following corollary.

\begin{prop}\label{rootcor}
For every rational $\frac{r}{s}$ the roots of the polynomials $\Rc$ and $\Sc$
belong to the open annulus $\{3-2\sqrt2< |t| < 3+2\sqrt2\}\subset\C^*$.
\end{prop}


\section{Proof of the main theorems}

The following important observation links the theory of $q$-rationals with Burau representation.

\begin{prop}\label{PROP}
The $q$-deformed action of the modular group (\ref{ActPSL}) coincides with the projective version of the Burau representation
(\ref{prho3}) with $q=-t.$
\end{prop}

\begin{proof}
Indeed, for $t=-q$ the matrices $R_q$ and $L_q$ given by (\ref{RqLq}) coincide with $\rho_3(\sigma_1)$ and $\rho_3(\sigma_2)^{-1}$.
\end{proof}

This means that from the point of view of Burau representation, $q$-rationals are the quotients
of the elements in column of the matrices 
$\rho_3(\b)$, for every $\b\in\B_3$, but the sign of the formal parameter is reversed.
More precisely, if the Burau representation written in the matrix form
\begin{equation}
\label{BurMat}
\rho_3(\b)=
\begin{pmatrix}\Rc(t)&\Vc(t)\\[2pt]
\Sc(t)&\Uc(t)
\end{pmatrix},
\end{equation}
then $\frac{\Rc(q)}{\Sc(q)}$ and $\frac{\Vc(q)}{\Uc(q)}$ with $q=-t$ are $q$-rationals.
For instance, taking $\b=\s_1\s_2^{-1}\s_1\s_2^{-1}$,
one obtains the matrix
$$
t^{-2}\begin{pmatrix}
-t+t^2-2t^3+t^4&1-t+t^2\\[4pt]
-t+t^2-t^3&1-t
\end{pmatrix}=
q^{-2}\begin{pmatrix}
q+q^2+2q^3+q^4&1+q+q^2\\[4pt]
q+q^2+q^3&1+q
\end{pmatrix},
$$
so that the rational functions
$\frac{1+q+2q^2+q^3}{1+q+q^2}$ and $\frac{1+q+q^2}{1+q}$
are $q$-deformed $\frac{5}{3}$ and $\frac{3}{2}$, respectively.

Now we can use the results of \cite{LMGOV} about the roots of the polynomials $\Rc,\Sc,\Vc$ and $\Uc$
to prove our main results. 

{

We start with the following lemma.

\begin{lem}
\label{L1}
Assume that in (\ref{BurMat}) matrix element $S(t)\equiv 0,$ then $\b$ belongs to the subgroup $G_0\subset \B_3$ generated by $\sigma_1$ and $(\sigma_1\sigma_2)^3.$
\end{lem}

\begin{proof} Specialization at $t=-1$ gives a matrix of the triangular form
$$\rho_3^{-1}(\b)=\begin{pmatrix} \pm 1&*\\[2pt]
0&\pm 1
\end{pmatrix} \in \PSL(2,\Z),
$$
which is the same as $\rho_3^{-1}(\sigma_1^k)$ for some $k \in \Z.$ Since the kernel of this specialization is the centre $Z$ of $\B_3$ 
generated by $(\sigma_1\sigma_2)^3,$ we have the claim. 
\end{proof}

Note the converse of the lemma is also obviously true.

Consider now the set $\Sigma \subset \C^*$ defined as the union of the poles of all $q$-rationals $\left[\frac{r}{s}\right]_q=\frac{\Rc(q)}{\Sc(q)}, \, \frac{r}{s} \in \mathbb Q.$ Note that since $$\left[\frac{1}{n}\right]_q=\frac{q^{n-1}(1-q)}{1-q^n}$$ $\Sigma$ contains all roots of unity except $1$ itself. Recall that $\Sigma_*=\Sigma \cup \{1\}$ is the extended set with added special point $q=1.$

We are ready to prove Theorem \ref{ThmMain1}, i.e. to prove that 
the Burau representation~$\rho_3$ specialized at $t_0\in  \C^{*}$ is faithful if and only if $-t_0 \notin \Sigma_*.$

\begin{proof}
Let us first prove that if $-t_0 \notin \Sigma_*$ then $\rho_3^{t_0}$ is faithful.
Given a braid $\beta\in\B_3$ we need to show that $\rho_3^{t_0}(\b)=\Id$ implies that $\b$ is a trivial braid.

Assume that $\rho_3^{t_0}(\b)=\Id$.
Proposition \ref{PROP} implies that the polynomials $ \Sc$ and~$ \Vc$ in the matrix~\eqref{BurMat}
vanish identically:  $\Sc(t)=\Vc(t)\equiv0$, since $t_0$ is not a root of these polynomials.
Therefore,
$$
\rho_3(\beta)=
\begin{pmatrix}
\Rc(t)&0\\[2pt]
0&\Uc(t)
\end{pmatrix}.
$$
We know that $\rho_3(\beta)$ evaluated at $t_0=-1$ belongs to $\SL(2,\Z)$.
Proposition~\ref{pos} then implies that
the polynomials $\Rc$ and $\Uc$ are monomials,
i.e., $\Rc(t)=\pm t^\ell$ and $\Uc(t)=\pm t^k$, for some integers $\ell$ and $k$.
Since, by assumption $t_0$ is not a root of $1$ or $-1$, $\rho_3^{t_0}(\b)=\Id$ implies
that the polynomials $\Rc$ and $\Uc$ are constant equal to~$1$.
Finally one has $\rho_3(\beta)=\Id$ and therefore $\b$ is a trivial braid because the Burau representation $\rho_3$ is faithful.

 To prove the converse statement we use some ideas from \cite{Sch}. We need to prove that if $-t_0 \in \Sigma_*$ then $\rho_3^{t_0}$ is unfaithful.  First of all, we know that this is true if $t_0=-1$. Assume now that  $-t_0 \in \Sigma$
and consider a braid $\b_0\in \B_3$ such the corresponding element $\Sc(t)$ in (\ref{BurMat}) is not identically zero, but $\Sc(t_0)=0.$
Assume that the specialization $\rho_3^{t_0}$ is faithful and consider the subgroup $G$ of $\B_3$ generated by $\sigma_1$ and $\b_0$.
Since both $\rho_3^{t_0}(\sigma_1)$ and $\rho_3^{t_0}(\b_0)$ are triangular, $G$ must be solvable, but not abelian (since $\b_0$ does not commute with $\sigma_1$ because $\Sc(t)\not\equiv0$).

Let us prove that this is impossible in our case.\footnote{We are very grateful to A.Yu. Olshanskiy for the help with this proof.} 
Introducing $x=\sigma_1\sigma_2$ and $y=\sigma_2\sigma_1\sigma_2$ we can rewrite the generating relation as $y^2=x^3.$ The centre $Z$ is generated by $x^3=y^2$, so the quotient $\B_3/Z=\Z_2*\Z_3$ is the free product of two cyclic groups of order 2 and 3 respectively (which is also isomorphic to the modular group $\PSL(2,\Z)$).
Consider the image $\phi(G) \subset \Z_2*\Z_3$ under the homomorphism $\phi: \B_3\to \B_3/Z.$ From Kurosh subgroup theorem \cite{Kurosh} it follows that any solvable subgroup of $\Z_2*\Z_3$ is either cyclic, or infinite dihedral group $D_\infty=\Z_2\ltimes \Z$, which is isomorphic to $ \Z_2*\Z_2$ (generated by the images of some conjugates of $y$). In the first case $G$ is abelian.
We claim that the second case is not possible as well. Indeed, if $G=D_\infty$, then it must belong to the kernel of the homomorphism $\psi: \B_3\to \Z_3$ sending $y\to 0$ and $x\to 1$ (in the additive notation). Since $\psi(\sigma_1)=\psi(x^2y^{-1})=2 \neq 0$, we have a contradiction.

Theorem \ref{ThmMain1} is proved.
\end{proof}

Now Theorem \ref{ThmMain2} is an immediate corollary of Theorem \ref{ThmMain1} and Proposition \ref{rootcor}.
Note also that Proposition~\ref{pos} implies the result of~\cite{Sch} about faithfulness of $\rho_3^t$ for negative~$t\not=-1$.

%
%
%

 
 \section{Discussion}


 Perhaps the most interesting question about specializations $\rho_3^t$
 is to understand the number-theoretic properties of the algebraic integers from the set $\Sigma$.  
 The following results from the theory of $q$-rationals can give more information in this direction.
 
 $\bullet$ The sequences of coefficients in the polynomials written in the parameter $q$ are unimodal.
 This fact, which was conjectured in~\cite{MGOfmsigma} and studied in~\cite{McCSS}, was eventually proved in~\cite{OgRa}.
 
This implies, in particular, that if an algebraic integer $\alpha\in \C$ has a conjugate
which is real and positive, then $\alpha$ cannot be part of $\Sigma$ 
(and hence, corresponding specialization of the Burau representation at $t_0=-\alpha$ is faithful).
 
  $\bullet$ The trace of the matrix~\eqref{BurMat} is a palindromic polynomial (in variable~$q$)~\cite{LMGadv}.
  We wonder if this is related to the invariant Hermitian form of~\cite{Squ}.

 $\bullet$ The ``stabilisation phenomenon'' of~\cite{MGOexp}
 is one of the main properties of $q$-rationals.
 Suppose that a sequence of rationals, $\frac{r_m}{s_m}$, converges to an irrational number $x$.
 Then the Taylor series of the rational functions $\frac{\Rc_m(q)}{\Sc_m(q)}$ stabilizes, as $m$ grows.
 This allows to define a $q$-deformation $[x]_q$ as Taylor series with integer coefficients.
 The study of the radii of convergence of these series is closely related to the study of the singular set $\Sigma$.
In particular if $\varphi=\frac{1+\sqrt{5}}{2}$ is the golden ratio, then the radius of convergence of the corresponding series 
 $[\varphi]_q$ is
$R(\varphi)=\frac{3-\sqrt{5}}{2}$ (see~\cite{LMGOV}).

The annulus in Theorem 2 can be reduced to $\frac{3-\sqrt5}{2}\leq |t| \leq \frac{3+\sqrt5}{2}$ modulo the following conjecture.


\begin{conj} \cite{LMGOV} 
 For every real $x>0$ the radius of convergence $R(x)$ of the series~$\left[x\right]_q$
satisfies the inequality
$$
R(x)\geq R(\varphi)=\frac{3-\sqrt{5}}{2}
$$
and the equality holding only for $x$ which are $\PSL(2,\Z)$-equivalent to~$\varphi$.
\end{conj}

Since for a $q$-rational $\left[\frac{r}{s}\right]_q=\frac{\Rc(q)}{\Sc(q)}$ the radius of convergence equals the minimal modulus of the roots of the denominator $\Sc(q)$, modulo this conjecture we can claim that $\Sigma_*$ lies within the annulus $\frac{3-\sqrt5}{2}\leq |q| \leq \frac{3+\sqrt5}{2}$ and therefore the specialized Burau representation $\rho_3^t$ is faithful for all $t$ outside the annulus $\frac{3-\sqrt5}{2}\leq |t| \leq \frac{3+\sqrt5}{2}.$ 

In some special cases Conjecture 2 was proved in~\cite{LMGOV, RenB}.
 Computer experiments show that the bounds $\frac{3\pm\sqrt5}{2}$ for the annulus in the
 conjectures are optimal.} 
 In particular, the polynomials in the entries of the matrices $M^+_q(1,1,\ldots, 1)=(R_qL_q)^m$ 
 have roots closer and closer to the circle $|t|=\frac{3-\sqrt5}{2}$, as $m$ grows.

%
%
%
%
%

As we have already mentioned in the Introduction, the Alexander polynomial of the link $L=L(\beta)$ related to a braid $\beta \in \B_n$ can be given (up to a unit in $\Z[t,t^{-1}]$) by the Burau formula \cite{Burau}
$$
\label{Alex}\Delta_L(t)=\frac{1-t}{1-t^n}\det(I-\rho_n(\beta)).
$$
No general formula in terms of Burau representation is known for the Jones polynomials. However, as it was noticed in~\cite{MGOfmsigma}, for the class of {\it rational (or two-bridge) knots} the (normalised) Jones polynomial
of the knot corresponding to a fraction~$\frac{r}{s}$
can be expressed using the polynomials $\Rc(q)$ and $\Sc(q)$ in~\eqref{RatRS} via
$$
J_{\frac{r}{s}}(q)=q\Rc(q)+(1-q)\Sc(q)
$$
(see Appendix A in \cite{MGOfmsigma}). The proof is based on the results of the paper  \cite{LeSc2} by Lee and Schiffler, using the combinatorics of snake graphs and cluster algebras.
 Alternatively, in terms of the so-called left version of the $q$-deformed rationals 
$$\left[\frac{r}{s}\right]^\flat_q=\frac{\Rc^\flat(q)}{\Sc^\flat(q)},$$
which was observed in~\cite{MGOexp} (cf. Remark 2.3) and
introduced and studied by Bapat et al \cite{BBL}, one can express the normalised Jones polynomial as $J_{\frac{r}{s}}(q)=q^m\Rc^\flat(q^{-1})$ with some $m\in \Z$ (see Appendix A2 in \cite{BBL} and \cite{Ren2}). It would be interesting to see if all this can be explained using Burau representation.
{ In this relation we would like to note that the role of the left $q$-rationals} became more clear after the recent preprint~\cite{Thom},
where an infinitesimal analogue of the Burau representation is investigated.


The Burau representation  $\rho_n$ fails to be faithful for $n\geq 5$, but there is another homological representation { of the braid group $\B_n$}, first introduced in \cite{Law} and now known as Lawrence-Krammer-Bigelow representation, which is proved to be faithful for all $n$ \cite{Big2,Kram}. The work of Lawrence \cite{Law} was initially motivated by the desire to better understand the significance of the Jones polynomial for links \cite{Jones}. Burau representation corresponds to the first non-trivial case $m=1$ in her construction. The general case is related also to the monodromy representations for the Knizhnik-Zamolodchikov equations
$$
\kappa \frac{\partial \psi}{\partial z_i}-\sum_{k\neq i}\frac{\Omega_{ik}}{z_i-z_k}\psi(z_1,\dots,z_n)=0, \quad i=1,\dots, n,
$$
 which were studied by Tsuchiya and Kanie \cite{TK} and Kohno \cite{Kohno}. 
 
 This makes an important link of the braid groups and knots with the theory of Yang-Baxter equations and quantum integrable systems (see more on this in \cite{KT,Turaev}). There is a well-known representation \cite{Jones} of the braid group in the Temperley-Lieb algebra, which appeared in a similar relation in statistical mechanics. Bigelow \cite{Big3} proved the equivalence of the faithfulness problem for Burau, Jones and Temperley-Lieb representations in the most interesting case $n=4.$ It is known that the Temperley-Lieb representation of the braid group $\B_3$ is faithful \cite{Piw}, but the same question about its specializations seems to be open. It would be interesting to apply our approach to this problem as well.
 
 Finally, we would like to mention that the specializations of Burau representations of $\B_n$ at the roots of unity naturally appear in algebraic geometry as the homological monodromy in the moduli space of algebraic curves
$$
y^d=x^n+a_1x^{n-1}+\dots + a_0.
$$

This observation goes back to the work of Arnold \cite{Arnold}, who considered the hyperelliptic case $d=2$, and was explicitly stated by Magnus and Peluso \cite{MaPe} in the general case. A more recent important development in this direction is due to Venkataramana \cite{Ven}, who proved the arithmeticity of the image of the Burau representation $\rho_n$ specialized at the $d$-th roots of unity when $n\geq 2d+1.$


\begin{thebibliography}{99}

\bibitem{Alexander} 
{\sc Alexander, J.W.} \newblock Topological invariants of knots and links.
\newblock {\em Trans. Amer. Math. Soc. }{\bf 30} (1928), 275-306.


 \bibitem{Arnold} 
{\sc Arnold, V.I.} \newblock Remark on the branching of hyperelliptic integrals as functions of the parameters.
\newblock {\em Funktsional. Anal. i Prilozhen. }{\bf 2} (1968), 1-3.


\bibitem{Artin} 
{\sc Artin, E.} \newblock Theorie der Z\"opfe.
\newblock {\em Abh. Math. Sem. Univ. Hamburg }{\bf 4} (1925), 47-72.


\bibitem{BBL}
{\sc Bapat, A., Becker, L., and Licata, A.}
\newblock $q$-deformed rational numbers and the 2-{C}alabi--{Y}au category of
  type $A_2$.
\newblock {\em Forum Math. Sigma 11\/} (2023), Paper No. e47.

\bibitem{Bigelow}
{\sc Bigelow, S.}
\newblock The Burau representation is not faithful for $n=5.$
\newblock {\em Geometry and Topology}, {\bf 3(1)} (1999), 397-404.

\bibitem{Big2}
{\sc Bigelow, S.}
\newblock Braid groups are linear.
\newblock {\em J. Amer. Math. Soc.} {\bf 14} (2001), 471-486.

\bibitem{Big3}
{\sc Bigelow, S.}
\newblock Does the Jones polynomial detect the unknot?
\newblock {\em J. Knot Theory Ramifications} {\bf 11} (2002), 493-505.


\bibitem{Bir}
{\sc Birman, J.} 
\newblock Braids, links, and mapping class groups. 
\newblock Ann. Math. Stud. 82. 
Princeton Univ. Press, 1974.

\bibitem{BB}
{\sc Bharathram, V., and  Birman, J.}
\newblock On the Burau representation of $B_3$, arXiv:2208.12378.

\bibitem{Burau} 
{\sc Burau, W.} \newblock \"Uber Zopfgruppen und gleichsinnig verdrillte Verkettungen.
\newblock {\em Abh. Math. Sem. Univ. Hamburg }{\bf 11} (1936), 179-186.

\bibitem{Jones} 
{\sc Jones, V.F.R.} \newblock Hecke algebra representations of braid groups and link polynomials.
\newblock {\em Ann. of Math. (2) }{\bf 126(2)} (1987), 335-388.

\bibitem{KT}
{\sc Kassel, Ch., Turaev, V.G.}
\newblock Braid Groups.
\newblock  Springer Verlag, 2008.

\bibitem{Kohno}
{\sc Kohno, T}
\newblock Hecke algebra representations of braid groups and classical Yang-Baxter equations. 
\newblock Adv. Stud. Pure Maths. {\bf 16} (1988), 255-269.

\bibitem{Kram}
{\sc Krammer, D.}
\newblock Braid groups are linear.
\newblock  {\em Ann. of Math. (2)} {\bf 155} (2002), 131-362.

\bibitem{Kurosh}
{\sc Kurosh, A.}
\newblock Die Untergruppen der freien Produkte von beliebigen Gruppen.
\newblock  {\em Math. Annalen} {\bf 109} (1934), 647-660.

\bibitem{Law}
{\sc Lawrence, R.J.}
\newblock Homological representations of the Hecke algebra.
\newblock {\em Comm. Math. Phys.}, {\bf 135} (1990), no. 1, 141-191.


\bibitem{LP}
{\sc Long, D.D. and Paton, M.}
\newblock The Burau representation is not faithful for $n\geq 6.$
\newblock Topology, {\bf 32(2)} (1992), 439-447.


\bibitem{LMGadv}
{\sc Leclere, L., and Morier-Genoud, S.}
\newblock {$q$}-deformations in the modular group and of the real quadratic
  irrational numbers.
\newblock {\em Adv. Appl. Math. 130\/} (2021).

\bibitem{LMGOV}
{\sc Leclere, L., Morier-Genoud, S., Ovsienko, V., and Veselov, A.}
\newblock On radius of convergence of $q$-deformed real numbers.
\newblock Mosc. Math. J. to appear, arXiv:2102.00891.

\bibitem{LeSc2}
{\sc Lee, K. and Schiffler, R.}
\newblock Cluster algebras and {J}ones polynomials.
\newblock {\em Selecta Math. (N.S.) 25}, 4 (2019).

\bibitem{MaPe}
{\sc Magnus, W., and Peluso, A.}
\newblock On a theorem of V. I. Arnol'd. 
\newblock {\em Comm. Pure Appl. Math. 22\/} (1969), 683--692. 

\bibitem{McCSS}
{\sc McConville, T., Sagan, B.~E., and Smyth, C.}
\newblock On a rank-unimodality conjecture of {M}orier-{G}enoud and {O}vsienko.
\newblock {\em Discrete Math. 344}, 8 (2021).

\bibitem{Moody-1991}
 {\sc Moody, J.A. } 
 \newblock The Burau representation of the braid group $B_n$ is unfaithful for large $n.$ 
\newblock {\em Bulletin of the American Mathematical Society 25(2)} (1991), 379-384.


\bibitem{MGOfmsigma}
{\sc Morier-Genoud, S., and Ovsienko, V.}
\newblock {$q$}-deformed rationals and {$q$}-continued fractions.
\newblock {\em Forum Math. Sigma 8\/} (2020), e13, 55 pp.

\bibitem{MGOexp}
{\sc Morier-Genoud, S., and Ovsienko, V.}
\newblock On $q$-deformed real numbers.
\newblock {\em Exp. Math. 31 (2022), 652--660.}

\bibitem{OgRa}
{\sc Oguz, E.~K., and Ravichandran, M.}
\newblock Rank polynomials of fence posets are unimodal.
\newblock {\em Discrete Math. 346}, 2 (2023), Paper No. 113218.
  
  \bibitem{Oamr}
{\sc Ovsienko, V.}
\newblock Modular invariant $q$-deformed numbers: first steps.
\newblock https://amathr.org/modular-invariant-q-deformed-numbers-first-steps/.

\bibitem{Piw}
{\sc Piwocki, A.H.}
\newblock  A new proof of the faithfulness of the Temperley-Lieb representation of $B_3.$
\newblock {\em Demonstratio Mathematica XLIV } (2011).

\bibitem{RenB}
{\sc Ren, X.}
\newblock  On radiuses of convergence of $q$-metallic numbers and related $q$-rational numbers. 
\newblock {\em Res. Number Theory 8} (2022).

\bibitem{Ren2}
{\sc Ren, X.}
\newblock On $q$-deformed Farey sum and a homological interpretation of $q$-deformed real quadratic irrational numbers.
\newblock arXiv:2210.06056.

\bibitem{Sch}
{\sc Scherich, N.}
\newblock Classification of the real discrete specializations of the Burau representation of $B_3$. 
\newblock Math. Proc. Camb. Phil. Soc., 168, 2020.


\bibitem{Squ}
{\sc Squier, C.}
\newblock The Burau representation is unitary. 
\newblock Proc. Amer. Math. Soc. 90 (1984), 199--202. 

\bibitem{Thom}
{\sc Thomas, A.}
\newblock  Infinitesimal modular group: $q$-deformed $\mathrm{sl}_2$ and Witt algebra.
\newblock arXiv:2308.06158.

\bibitem{TK}
{\sc Tsuchiya, A., and Kanie, Y.}
\newblock Vertex operators in conformal field theory on $\mathbb P^1$ and monodromy representations of braid groups.
\newblock Adv. Stud. Pure Math. {\bf16} (1988), 297-372.

\bibitem{Turaev}
{\sc Turaev, V.G.}
\newblock The Yang-Baxter equation and invariants of links. 
\newblock Invent. Math. {\bf 92} (1988), 527-553.

\bibitem{Ven}
{\sc Venkataramana, T.N.}
\newblock Image of the Burau representation at $d$-th roots of unity. 
\newblock  Annals of Mathematics {\bf 179} (2014), 1041-1083.

\end{thebibliography}
\end{document}